\documentclass[smallextended]{svjour3}     
\smartqed  

\usepackage[
type={CC},
modifier={by},
version={4.0},
]{doclicense}

\usepackage{graphicx}
\usepackage{prodint}
\usepackage{slashbox}
\usepackage{subcaption}

\usepackage{multirow}
\usepackage{enumerate}
\usepackage[usenames]{color}

\usepackage{amssymb} 
\usepackage{amsthm}
\usepackage{amsmath}
\usepackage{verbatim}

\journalname{xxx}

\newtheorem{defi}{\noindent Definition}
\newtheorem{coro}{\noindent Corollary}
\newtheorem{prop}{\noindent Proposition}

\begin{document}

\title{Two stochastic versions of the Arps curve decline}
\author{Ch. Paroissin}
\institute{
Christian Paroissin \at 
Universite de Pau et des Pays de l’Adour, E2S UPPA, CNRS, LMAP, Pau, France.
\email{christian.paroissin@univ-pau.fr}
}
\date{Last version: \today}

\maketitle

\doclicenseThis

\begin{abstract} 
Based on the Arps equation, we propose two stochastic models for curve decline useful in oil engineering context. Theoretical properties and simulations of these models are provided. The first passage time distribution of these stochastic models to a constant level is then studied. In conclusion, we discuss about statistical inference of the parameters from the observations of the oil production cumulative rate.

\end{abstract}
	
\section{Introduction}

Arps \cite{Arps} (see also Chapter~9 in \cite{LW})) introduced several decades ago an empirical decline equation during pseudo-steady state period. This curve decline is of the following form:
\begin{equation}\label{eqn:arps}
\forall t \geqslant 0 , \quad q_t = \frac{q_0}{(1+bd_0t)^{1/b}},
\end{equation}
where $q_t$ is the oil production rate at production time $t$ and $q_0$ is the initial oil production rate and where $b$ and $d_i$ are two constants representing respectively a shape parameter and the initial flow. The case $b=0$ represents an exponential decline:
$$
\forall t \geqslant 0 , \quad q_t = q_0 e^{-d_0t},
$$
and $b=1$ represents a harmonic decline in oil production. In other cases ($0<b<1$), Equation~(\ref{eqn:arps}) is defined as the hyperbolic model. Although this decline curve proposed by Arps is based only on an empirical study, this equation is still widely used in oil industry studies (see \cite{DeanMireault} or \cite{GaskariMohagheghJalali}, for instance). 

Here we propose stochastic models for curve decline analysis (section 2). These models are based on stochastic differential equations such that the expectation is the Arps equation. Uncertainty in decline curve analysis were already considered. For instance, Chang and Lin \cite{ChangLin} have proposed a regression linear analysis of the decline curve after applying a transformation to fit data. We then study the important question of how long will the oil production rate be larger than a certain level (section 3). Such problem is difficult and no exact computations can be carried out (mainly numerical ones). In conclusion, we discuss about statistical estimation of the parameters. Most of the time the oil production cumulative rate is observed.

\section{Stochastic Arps model}

We introduce here a set of stochastic differential equations (SDEs) modelling curve  decline for oil production for instance. These SDEs are of the form:
\begin{equation}\label{eqn:gsde}
\mathrm{d}Q_t = -\frac{d_0Q_t}{1+bd_0t} \mathrm{d}t + \alpha(Q_t)\mathrm{d}B_t ,
\end{equation}
with initial condition $Q_0=q_0$ and where $(B_t)$ is a Brownian motion and $\alpha(\cdot)$ a real function. These SDEs involve the two parameters $d_0$ and $b$ that appear in the Arps equation. 

When $b$ tends to zero, both the drift and the volatility do not depend no more on time (only on space); in such case the corresponding diffusion process is said to be homogeneous. In the general setting, if the function $\alpha(\cdot)$ satisfies the following assumptions:
\begin{itemize}
\item[$(A_1)$] (Lipschitz condition) there exists a constant $\kappa_1>0$ such that:
$$
\forall t \in [0,T], \forall (x,y) \in \mathbb{R}^2 , \quad | \alpha(x)-\alpha(y)| \leqslant \kappa_1 |x-y|,
$$ 
\item[$(A_2)$] (linear growth bound) there exists a constant $\kappa_2>0$ such that:
$$
\forall t \in [0,T], \forall x \in \mathbb{R} , \quad | \alpha(x)| \leqslant \kappa_2^2 (1+|x|^2),
$$
\end{itemize}
then the corresponding SDE admits an unique strong solution on $[0,T]$, applying theorem~4.5.3 in \cite{KP} for instance. Two special cases, for which the above conditions are clearly satisfied, will be studied below. The first one is when the volatility is constant and the second is a case where the volatility is linearly increasing. It follows that these two special cases leads to scalar linear SDEs (see \cite{KP}, page~110).

\subsection{Case with constant volatility}

In this subsection we will assume that the volatility is constant: $\alpha(x)=\sigma>0$. Thus Equation~(\ref{eqn:gsde}) turns to be:
\begin{equation}\label{eqn:sde1}
\mathrm{d}Q_t = -\frac{d_0Q_t}{1+bd_0t} \mathrm{d}t + \sigma \mathrm{d}B_t,
\end{equation}
This choice of the function $\alpha(\cdot)$ can be motivated as follows: this SDE can be interpreted as a stochastic perturbation of the ordinary differential equation satisfied by $q_t$. 

When the parameter $b$ tends to 0, the SDE is the so-called Ornstein-Uhlenbeck process \cite{Oksendal}. The solution is then given by:
\begin{equation}\label{eqn:sol.ou}
\forall t \geqslant 0 , \quad Q_t = q_0 e^{-d_0t} + \sigma e^{-d_0t} \int_0^t e^{d_0u}\,\mathrm{d}B_u.
\end{equation}
As for this special case, the solution can be computed:

\begin{prop}
The solution of Equation~(\ref{eqn:sde1}) is:
\begin{equation}\label{eqn:sol.sde1}
\forall t \geqslant 0 , \quad Q_t = q_0 (1+bd_0t)^{-1/b} + \sigma (1+bd_0t)^{-1/b} \int_0^t (1+bd_0u)^{1/b}\,\mathrm{d}B_u .
\end{equation}
\end{prop} 

\begin{proof}
For all $t\geqslant 0$, set $X_t = (1+bd_0t)^{1/b}Q_t$ with initial value $X_0=Q_0=q_0$. It follows that:
\begin{eqnarray*}
dX_t 
 & = & d_0(1+bd_0t)^{1/b-1}Q_t\mathrm{d}t + (1+bd_0t)^{1/b} \mathrm{d}Q_t \\
 & = & \sigma (1+bd_0t)^{1/b} \mathrm{d}B_t .
\end{eqnarray*}
Thus it follows that:
$$
X_t = X_0 + \sigma \int_0^t (1+bd_0u)^{1/b} \mathrm{d}B_u.
$$
Back to the initial process, one obtains the following solution of the SDE:
$$
\forall t \geqslant 0 , \quad Q_t = q_0 (1+bd_0t)^{-1/b} + \sigma (1+bd_0t)^{-1/b} \int_0^t (1+bd_0u)^{1/b}\,\mathrm{d}B_u.
$$
\end{proof}

The integral in the solution is a Wiener one, it follows that $(Q_t)$ is a Gaussian process with mean:
$$
\forall t \geqslant 0 , \quad \mathbb{E}[Q_t] = q_0 (1+bd_0t)^{-1/b} = q_t
$$
and variance:
$$
\forall t \geqslant 0 , \quad {\mbox{Var}}[Q_t] = \sigma^2 (1+bd_0t)^{-2/b} \int_0^t (1+bd_0u)^{2/b}\,\mathrm{d}u 
= \frac{\sigma^2(1+bd_0t)}{d_0(2+b)}\left(1-(1+bd_0t)^{-2/b-1}\right) .
$$
In addition, the covariance between $Q_s$ and $Q_t$, for any $(s,t) \in \mathbb{R}_+^2$, is given by :
$$
{\mbox{Cov}}[Q_s, Q_t] = \sigma^2 (1+bd_0s)^{-1/b} (1+bd_0t)^{-1/b} \int_0^{\min(s,t)} (1+bd_0u)^{2/b}\,\mathrm{d}u .
$$
It follows that the expectation of this SDE is precisely the Arps equation. When considering Equation~(\ref{eqn:sol.sde1}) and with $b=0$, one recovers the solution given in Equation~(\ref{eqn:sol.ou}) for the Ornstein-Uhlenbeck process.

One can express $(Q_t)$ as a Brownian motion with a time scale change. Indeed applying the Dambis-Dubins-Schwartz theorem \cite{RevuzYor}, there exists a Brownian motion $(W_t)$ such that:
$$
Z_t = \int_0^t (1+bd_0u)^{1/b}\,\mathrm{d}B_u \stackrel{(d)}{=} W_{\tau(t)} ,
$$
where $\stackrel{(d)}{=}$ stands for equality in distribution and where:
$$
\tau(t)={\mbox{Var}}[Z_t] = \frac{1}{d_0(b+2)} \left[ \left(1+bd_0t\right)^{\tfrac{2}{b}+1}-1 \right].
$$
This expression is the same as the one in \cite{APP} when the parameter $b$ tends to zero. In any case, as a consequence, the following representation holds:
\begin{equation}\label{eqn:dds}
Q_t \stackrel{(d)}{=} (1+bd_0t)^{-1/b} (q_0+\sigma W_{\tau(t)}) .
\end{equation}
On Figure \ref{fig:sde1} we present simulations of this SDE for various values of the parameter $b$ and with $d_0 = 3 \times 10^{-4}$, $q_0 = 380$ and $\sigma^2 = 1$. To do so, we have considered the Euler–Maruyama method (see \cite{KP}, for instance). We have used the same simulations of the normal distribution in order to compare the different paths. For fixed value of $t$, $q_t$ is decreasing with parameter $b$. Hence, using the same simulations of normally distributed random variables, it follows that, for any $t$, sample paths are increasing with $b$.
\begin{figure}[htp!]
\begin{center}
\includegraphics[width=12cm]{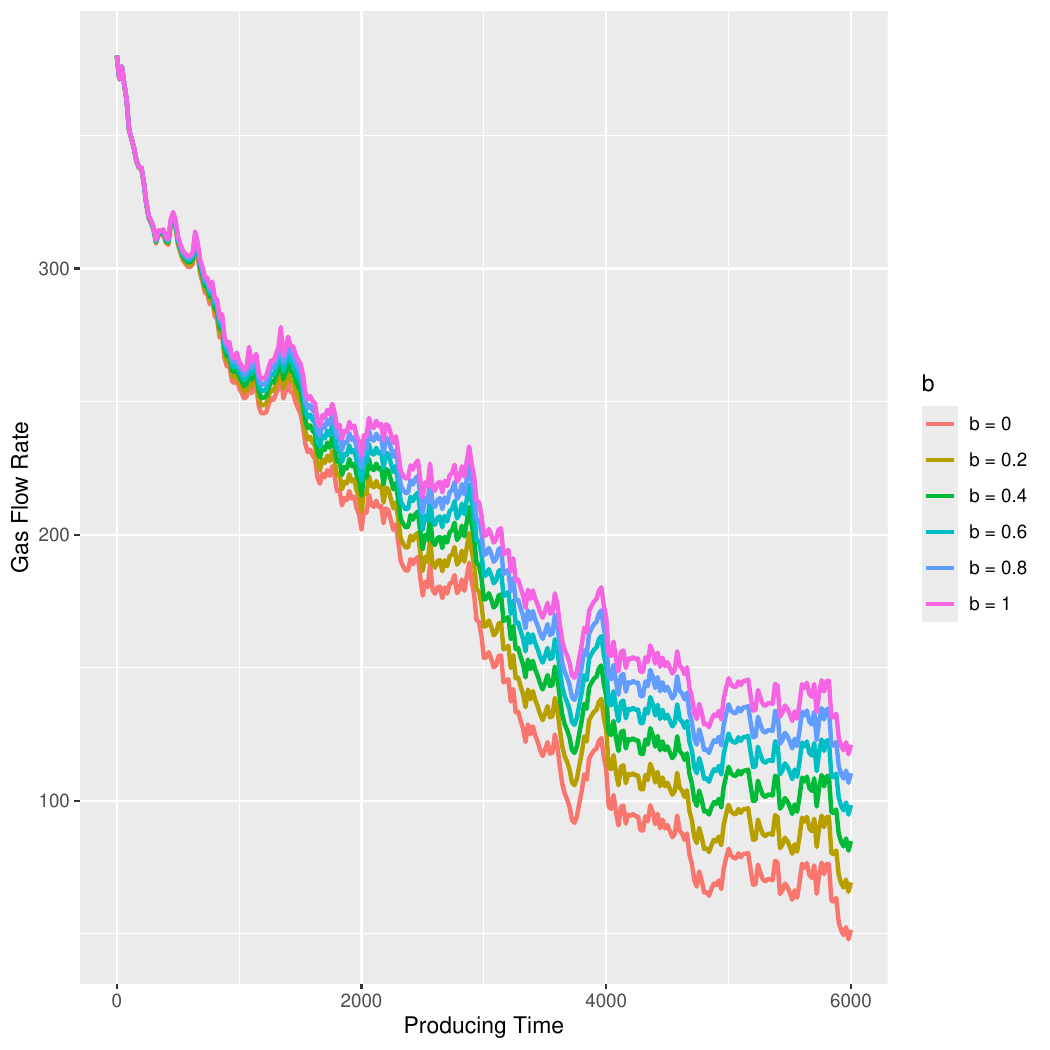}
\end{center}
\caption{Simulations of the first stochastic models}
\label{fig:sde1}
\end{figure}

\subsection{Case with linear volatility}

Now we will consider a case where the volatility is increasing. More precisely we assume that $\alpha(x)=\sigma x$ with $\sigma>0$. Thus Equation~(\ref{eqn:gsde}) turns to be:
\begin{equation}\label{eqn:sde2}
\mathrm{d}Q_t = -\frac{d_0Q_t}{1+bd_0t} \mathrm{d}t + \sigma Q_t\mathrm{d}B_t ,
\end{equation}
This choice of the function $\alpha(\cdot)$ can be motivated as follows: this SDE can be interpreted as a stochastic perturbation of the ordinary differential equation satisfied by $\ln(q_t)$. 

When the parameter $b$ tends to 0, the SDE is the geometric Brownian motion \cite{Oksendal} and the solution is given by:
\begin{equation}\label{eqn:sol.gbm}
\forall t \geqslant 0 , \quad Q_t = q_0 e^{-d_0t} \exp\left( -\frac{1}{2}\sigma^2 t + \sigma B_t \right) .
\end{equation}
As for this special case, the solution can be computed:

\begin{prop}
The solution of Equation~(\ref{eqn:sde2}) is:
\begin{equation}\label{eqn:sol.sde2}
\forall t \geqslant 0 , \quad Q_t = q_0 (1+bd_0t)^{-1/b} \exp\left( -\frac{1}{2}\sigma^2 t + \sigma B_t \right) .
\end{equation}
\end{prop} 

\begin{proof}
Let us divide Equation~(\ref{eqn:sde2}) by $Q_t$. It gives:
$$
\frac{\mathrm{d}Q_t}{Q_t} = -\frac{d_0}{1+bd_0t} \mathrm{d}t + \sigma \mathrm{d}B_t .
$$
It follows that:
\begin{equation}\label{eqn:1}
\int_0^t \frac{\mathrm{d}Q_u}{Q_u} = -\frac{1}{b}\ln(1+bd_0t) + \sigma B_t .
\end{equation}
Thus one has to evaluate the left hand-side of this equation. Let us apply the Itô formula \cite{Oksendal} to $\ln Q_t$:
$$
\mathrm{d} \ln Q_t = \frac{\mathrm{d}Q_t}{Q_t} - \frac{1}{2}\sigma^2 \mathrm{d}t
$$
Integrating this equation, one gets:
\begin{equation}\label{eqn:2}
\int_0^t \frac{\mathrm{d}Q_u}{Q_u} = \ln\left(\frac{Q_t}{q_0}\right) + \frac{1}{2}\sigma^2 t . 
\end{equation}
Joining Equations (\ref{eqn:1}) and (\ref{eqn:2}), it gives:
$$
\ln\left(\frac{Q_t}{q_0}\right) = -\frac{1}{b}\ln(1+bd_0t) - \frac{1}{2}\sigma^2 t + \sigma B_t .
$$
Finally one gets:
$$
Q_t = q_0 (1+bd_0t)^{-1/b} \exp\left( -\frac{1}{2}\sigma^2 t + \sigma B_t \right) .
$$
\end{proof}

Clearly, for any $t$, $Q_t$ has the log-normal distribution with parameters $q_t - \tfrac{1}{2}\sigma^2 t$ and $\sigma^2 t$. Thus one gets that:
$$
\forall t \geqslant 0 , \quad \mathbb{E}[Q_t] = q_0 (1+bd_0t)^{-1/b}
$$
and variance:
$$
\forall t \geqslant 0 , \quad {\mbox{Var}}[Q_t] = q_0^2 (1+bd_0t)^{-2/b}\left( e^{\sigma^2 t}-1\right)   .
$$
Thus, as for the previous model, the expectation of this SDE is the Arps equation.

On Figure \ref{fig:sde2} we present the simulation of this SDE for various values of the parameter $b$ and with $d_0 = 3 \times 10^{-4}$, $q_0 = 380$ and $\sigma^2 = 10^{-2}$. As previously, we have considered the Euler–Maruyama method \cite{KP} and we have used the same simulations of the normal distribution in order to compare the different paths.
\begin{figure}[htp!]
\begin{center}
\includegraphics[width=12cm]{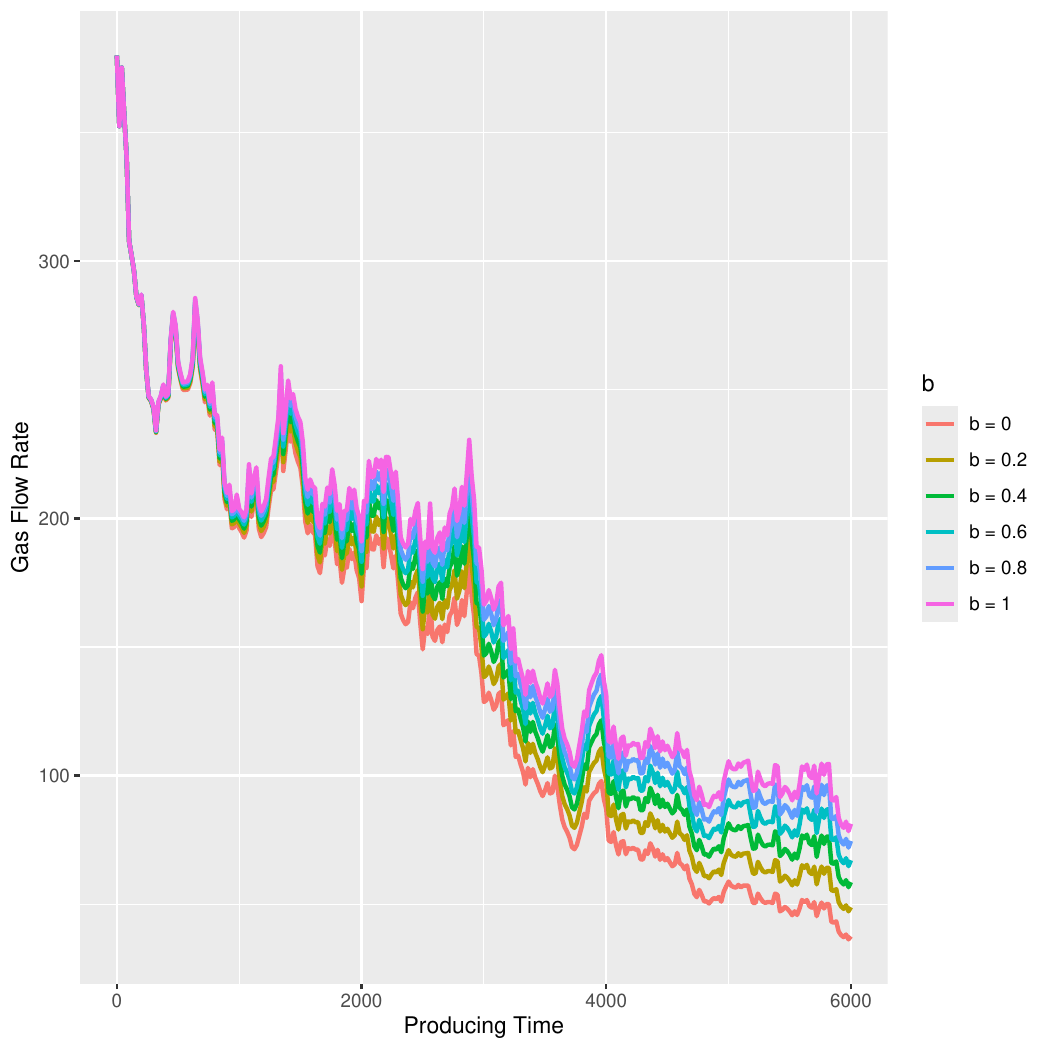}
\end{center}
\caption{Simulations of the second stochastic models}
\label{fig:sde2}
\end{figure}

\section{Hitting time problem}

An important question that could raise in petroleum context is the following one: when the production rate at a given site will be lower than a certain threshold? Thus one can be interested in the hitting time $T_{Q,x}$ of level $x<q_0$ for the stochastic process $Q=(Q_t)$:
$$
T_{Q,x} = \inf \{ t ; Q_t \leqslant x \}  = \inf \{ t ; Q_t = x \} ,
$$
since $Q$ has continuous sample paths. Such random variable has been well studied in the case of exponential decline ($b=0$) or harmonic decline ($b=1$) since it corresponds to some classical processes: the Ornstein-Ulhenbeck process (see \cite{APP} or \cite{Linetsky} for instance) in the first case and the geometric Brownian motion in the second one. In fact, for any value of $b$, the problem can be reduced to the first-passage time of a Brownian motion through a one-sided moving barrier. Thus the first subsection is devoted to a (non-exhaustive) review of some results about it. In addition, we provide a lemma that will be used later. Then in the two last subsections we discuss about the hitting time distributions for our problem. In particular we will compare fist passage time distribution in function of the parameter $b$. To this purpose we will also denote $T_{Q,x}(b)$ instead of $T_{Q,x}$.

\subsection{Brownian motion and hitting times}

The first-passage time of a Brownian motion through a boundary is not an easy problem and the distribution of this random variable is not known explicitly except for a constant or a linear barrier. In these cases the hitting time distribution is called the inverse Gaussian distribution.

Let $(B_t)$ be a standard Brownian motion and $c(\cdot)$ a real function. Consider the first-passage time of $(B_t)$ through the boundary $c(\cdot)$:
$$
T_{B,c(\cdot)} = \inf \{ t ; B_t = c(t) \} .
$$
If the $c(\cdot)$ is constant or linear, then $T_{B,c(\cdot)}$ has the inverse Gaussian distribution \cite{FC}. More precisely, if $c(t)=at+b$, then the probability distribution function of $T_{B,c(\cdot)}$ is given by:
$$
f(t) = \frac{|b|}{\sqrt{2\pi t^3}} \exp \left( - \frac{(at+b)^2}{2t} \right).
$$
It corresponds to the inverse Gaussian distribution with parameters $(|b|/a,b^2)$. The expectation of $T_{B,c(\cdot)}$ is equal to $|b|/a$ and the variance to $|b|/a^3$. The Laplace transform of such probability distribution has been obtained by Abdel-Hameed and Nakhi \cite{AH-Nakhi} (see also \cite{AH:survey}, page 19):
$$
{\cal{L}}_{T_{B,c(\cdot)}}(s) = \mathbb{E}[e^{-sT_{B,c(\cdot)}}] = \exp\left( -|b|\left(\sqrt{a^2 + 2s}-a\right) \right).
$$
Based on these two special cases (constant and linear boundaries), several authors study the case of constant piecewise and linear piecewise boundaries (see \cite{Abundo1} and references therein). A general expression (not necessary leading to an explicit formula, excepted for some cases) was given by Salminen \cite{Salminen} in the case where $c(\cdot)$ is twice continuously differentiable on $[0,+\infty)$.  

We will compare hitting time distribution according to parameters values (especially the parameter $b$). We first recall the definition of the stochastic comparison in the usual sense \cite{MullerStoyan}:
\begin{defi}
The random variable $X$ is said to be stochastically smaller than te random variable $Y$ if and only if:
$$
\forall t \in \mathbb{R} , \quad \mathbb{P}[X>t] \leqslant \mathbb{P}[Y>t].
$$
It will denoted as follows: $X \preceq_{st} Y$.
\end{defi}

We recall that $X \preceq_{st} Y$ implies that $\mathbb{E}[X] \leqslant \mathbb{E}[Y]$. Now we give the lemma that will be used later (proofs are omitted since there is no real difficulty):

\begin{lemma}\label{lem:stoch.comp}
\hspace{1cm}
\begin{enumerate}
\item If $X$ and $Y$ are two positive random variables defined on the same probability space such that $X \preceq_{st} Y$ and if $f$ is a one-to-one map defined on $\mathbb{R}^+$, then $f(X) \preceq_{st} f(Y)$.
\item If $f$ and $g$ are two functions such that that for any $t \in \mathbb{R}$ $f(t) \leqslant g(t)$, then for any random variable $X$ $f(X) \preceq_{st} g(X)$.
\item Let $f$ and $g$ two continuous functions defined on $\mathbb{R}^+$ such that:
$$
\forall t \geqslant 0 , \quad f(t) \geqslant g(t) ,
$$
with $g(0) \leqslant f(0) \leqslant 0$. Then $T_{B,f(\cdot)} \preceq_{st} T_{B,g(\cdot)}$.
\end{enumerate}
\end{lemma}

\subsection{Case with constant volatility}

In the case of exponential decline, the hitting time of a fixed value has been studied recently. There is no closed form expression for the probability distribution function \cite{APP}. However expressions were obtained that yield numerical computations. For instance, the Laplace transform of $T_{Q,x}$ can be expressed in term of Hermite functions $H_\nu$ or in term of parabolic cylinder functions $D_\nu$ \cite{Breiman,Siegert}:
$$
\mathbb{E}\left[e^{-uT_{Q,x}}\right] = \frac{H_{-u/d_0}\left( \tfrac{q_0}{\sigma}\sqrt{d_0} \right)}{H_{-u/d_0}\left( \tfrac{x}{\sigma}\sqrt{d_0} \right)} = \frac{e^{d_0q_0^2/2\sigma^2}}{e^{d_0x^2/2\sigma^2}} \frac{D_{-u/d_0}\left( \tfrac{q_0}{\sigma}\sqrt{2d_0} \right)}{D_{-u/d_0}\left( \tfrac{x}{\sigma}\sqrt{2d_0} \right)} .
$$
Then one can get numerically the p.d.f. of $T_{Q,x}$ applying for an appropriate algorithm (see \cite{lt} for instance). Otherwise a way to compute the hitting time distribution is to consider the representation of $(Q_t)$ by a time rescaled Brownian motion. The following lemma states the relationship between these two hitting time problems:

\begin{prop}
For any $t \geqslant 0$, let $c_1(t) = \frac{q_0}{\sigma} \left( \frac{x}{q_{\tau^{-1}(t)}}-1 \right)$. Let $W$ be a Brownian motion. Then $T_{Q,x}$ and $\tau^{-1}(T_{W,c_1(\cdot)})$ have the same distribution.
\end{prop}

\begin{proof}
It is an application of the representation of $(Q_t)$ given by Equation~(\ref{eqn:dds}). Indeed, for any $t \geqslant 0$, one gets:
\begin{eqnarray*}
\mathbb{P}[T_{Q,x} > t]
 & = & \mathbb{P}\left[ \forall u \in [0,t], Q_u \geqslant x \right] \\
 & = & \mathbb{P}\left[ \forall u \in [0,t], q_u \left(q_0+\frac{\sigma}{q_0} W_{\tau(u)}\right) > x\right] \\
 & = & \mathbb{P}\left[ \forall r \in [0,\tau(t)], q_{\tau^{-1}(r)} \left(q_0+\frac{\sigma}{q_0} W_r\right)> x\right] \\
 & = & \mathbb{P}\left[ \forall r \in [0,\tau(t)], W_r > \frac{q_0}{x}\left( \frac{x}{q_{\tau^{-1}(r)}}-1\right) \right] \\
 & = & \mathbb{P}[T_{W,c_1(\cdot)} > \tau(t)] \\
 & = & \mathbb{P}[\tau^{-1}(T_{W,c_1(\cdot)}) > t] .
\end{eqnarray*}
using that $\tau$ is one-to-one map with $\tau(0)=0$.
\end{proof}

From the expression of $\tau$ given in the previous, one can easily compute the inverse function:
$$
\forall t \geqslant 0 , \quad \tau^{-1}(t)= \frac{1}{bd_0} \left[ (1+t(b+2)d_0)^{\tfrac{2}{b+2}} -1 \right] .
$$
Hence the boundary involved in the previous lemma is:
$$
c_1(t) = \frac{q_0}{\sigma} \left( \frac{x}{q_{\tau^{-1}(t)}}-1 \right)
= \frac{1}{\sigma}\left[ x\left(1+d_0(b+2)t\right)^{\tfrac{1}{b+2}} - q_0 \right] .
$$
Notice that $c_1(0)=(x-q_0)/\sigma<0$. As $b$ tends to zero, $c_1$ turns to be a square-root boundary:
$$
c_1(t) = \frac{1}{\sigma}\left[ x\sqrt{1+2d_0t} - q_0 \right] .
$$

We now compare stochastically the hitting time distributions according to the value of the parameter $b$:

\begin{prop}\label{prop:stcomp.cste}
If $b_1 \leqslant b_2$, then $T_{Q,x}(b_1) \preceq_{st} T_{Q,x}(b_2)$.
\end{prop}

\begin{proof}
Let $b_1 \leqslant b_2$. Applying the third property of Lemma~\ref{lem:stoch.comp}, it follows that
$$
T_{W,c_1(\cdot)}(b_1) \preceq_{st} T_{W,c_1(\cdot)}(b_2).
$$
Now, using the first property in this Lemma,
$$
T_{Q,x}(b_1) \stackrel{(d)}{=} \tau^{-1}_{b_1}(T_{W,c_1(\cdot)}(b_1)) \preceq_{st} \tau^{-1}_{b_1}(T_{W,c_1(\cdot)}(b_2)),
$$
and the second property still in Lemma~\ref{lem:stoch.comp} gives
$$
\tau^{-1}_{b_1}(T_{W,c_1(\cdot)}(b_2)) \preceq_{st} \tau^{-1}_{b_2}(T_{W,c_1(\cdot)}(b_2)) \stackrel{(d)}{=} T_{Q,x}(b_2) .
$$
Joining these stochastic inequalities lead to the result stated in the proposition.
\end{proof}

This is consistent with the disadvantages pointed out in the literature when using the Arps equation: the exponential decline ($b=0$) leads to an underestimation of reserves and production rates while the harmonic decline ($b=1$) leads to an overestimation of reservoir performance \cite{Khanamiri,LiHorne}.

Using the above proposition with $b_1=0$ towards with the expression for the mean first-passage time obtained by Sato (\cite{Sato}, see also \cite{RicciardiSato}), we have the following corollary (using the same notations as in \cite{Sato}):
\begin{coro}
The mean hitting time for $(Q_t)$ to reach the level $x<q_0$ is bounded as follows:
$$
\mathbb{E}[T_{Q,x}] \geqslant \phi^{(1)}(0,x) - \phi^{(1)}(0,q_0),
$$
where:
$$
\phi^{(1)}(0,z) = \frac{1}{d_0} \left( \frac{z\sqrt{d_0 \pi}}{\sigma} F\left( \frac{1}{2},\frac{3}{2},\frac{d_0z^2}{\sigma^2} \right) + \sum_{m=0}^{\infty} \frac{2^m}{(m+1)(m+2)!!} \left( \frac{z\sqrt{d_0}}{\sigma} \right)^{2m+2} \right). 
$$
in which $F$ denotes the Kummer function (or confluent hypergeometric function).
\end{coro}

Some approximations of the mean first-passage time of a fixed level by the Ornstein-Uhlenbeck process have been studied by Thomas \cite{Thomas}.

\subsection{Case with linear volatility}

The exponential case, i.e. with $b=0$, is easy to derive. In fact, in such case, $T_{Q,x}$ is the hitting time of a constant by a Brownian motion with a linear drift. More precisely, after some easy computations, one gets:
$$
T_{Q,x} = \inf \left\{ t ; -(d_0+\tfrac{1}{2}\sigma^2)t + \sigma B_t \geqslant \ln\left(\frac{x}{q_0}\right) \right\}.
$$
It follows that $T_{Q,x}$ has the inverse Gaussian distribution with parameter:
\begin{equation}\label{eqn:ig.gmb.param}
m= \frac{2}{2d_0+\sigma^2} \ln\left(\frac{q_0}{x}\right) \quad {\mbox{and}} \quad 
\lambda= \left( \frac{1}{\sigma} \ln\left(\frac{q_0}{x}\right) \right)^2 .
\end{equation}
Let us now consider the general case: $T_{Q,x}$ is the hitting time by a Brownian motion of barrier of the form:
\begin{eqnarray*}
c_2(t) 
 & = & \frac{1}{2}\sigma t - \frac{1}{\sigma} \ln\left(\frac{q_t}{x} \right) \\
 & = & \frac{1}{2}\sigma t + \frac{1}{\sigma} \ln\left(\frac{x}{q_0} \right) +\frac{1}{b\sigma} \ln\left(1+bd_0t \right) .\\
\end{eqnarray*}
Using our notations we have $T_{Q,x} \stackrel{(d)}{=} T_{B,c_2(\cdot)}$. The boundary $c_2(\cdot)$ can bounded as follows by a linear boundary:
$$
\forall t \geqslant 0 , \quad c_2(t) \leqslant \tilde{c}_2(t) =  \frac{1}{\sigma} \ln\left(\frac{x}{q_0} \right) + \left(\frac{1}{2}\sigma +\frac{d_0}{\sigma} \right) t,
$$
the equality holding if $b=0$. Notice that $c_2(0)=\tilde{c}_2(0)=\log(x/q_0)/\sigma<0$. More generally, using the third property of Lemma~\ref{lem:stoch.comp}, the following proposition holds:
\begin{prop}\label{prop:stcomp.lin}
If $b_1 \leqslant b_2$, then $T_{Q,x}(b_1) \preceq_{st} T_{Q,x}(b_2)$.
\end{prop}

As we recall previously, this proposition implies that $\mathbb{E}[T_{Q,x}(b_1)] \leqslant \mathbb{E}[T_{Q,x}(b_2)]$. In particular, $T_{Q,x}(0)$ is inverse Gaussian distributed with parameters given in Equation~(\ref{eqn:ig.gmb.param}). So we have:
$$
\mathbb{E}[T_{Q,x}] \geqslant  \frac{2 \ln(q_0/x)}{2d_0 + \sigma^2}.
$$
This is also coherent with the disadvantages pointed out in the literature as we discussed above for the first stochastic model. This lower bound tends to the hitting time of level $x$ in the deterministic exponential curve decline as the parameter $\sigma^2$ tends to zero.

Since the function $c_2(\cdot)$ is continuously differentiable, $T_{Q,x}$ has a continuous probability density function \cite{Ferebee}. The exact probability density function $f_{T_{Q,x}}(t)$ of $T_{Q,x}$ cannot be compute explicitly but can be approximated using the result by Durbin \cite{Durbin}. Indeed $f_{T_{Q,x}}(t)$ is approximatively equal to $p(t)f(t)$ with:
$$
p(t) = \frac{1}{\sigma t}\ln \left(\frac{q_0}{x}\right) - \frac{1}{b \sigma t}\ln(1+bd_0t) - \frac{d_0}{\sigma(1+bd_0t)} 
$$
and
$$
f(t) = \frac{1}{\sqrt{2\pi t}} \exp\left( -\frac{c_2(t)^2}{2t} \right).
$$
For various values of $b$ and with $d_0 = 3 \times 10^{-4}$, $q_0 = 380$, $\sigma^2 = 10^{-2}$ and $x = 100$, we have computed numerically the approximation of $f_{T_{Q,x}}(t)$ which are plotted on the Figure \ref{fig:ht2}.
\begin{figure}[htp!]
\begin{center}
\includegraphics[width=12cm]{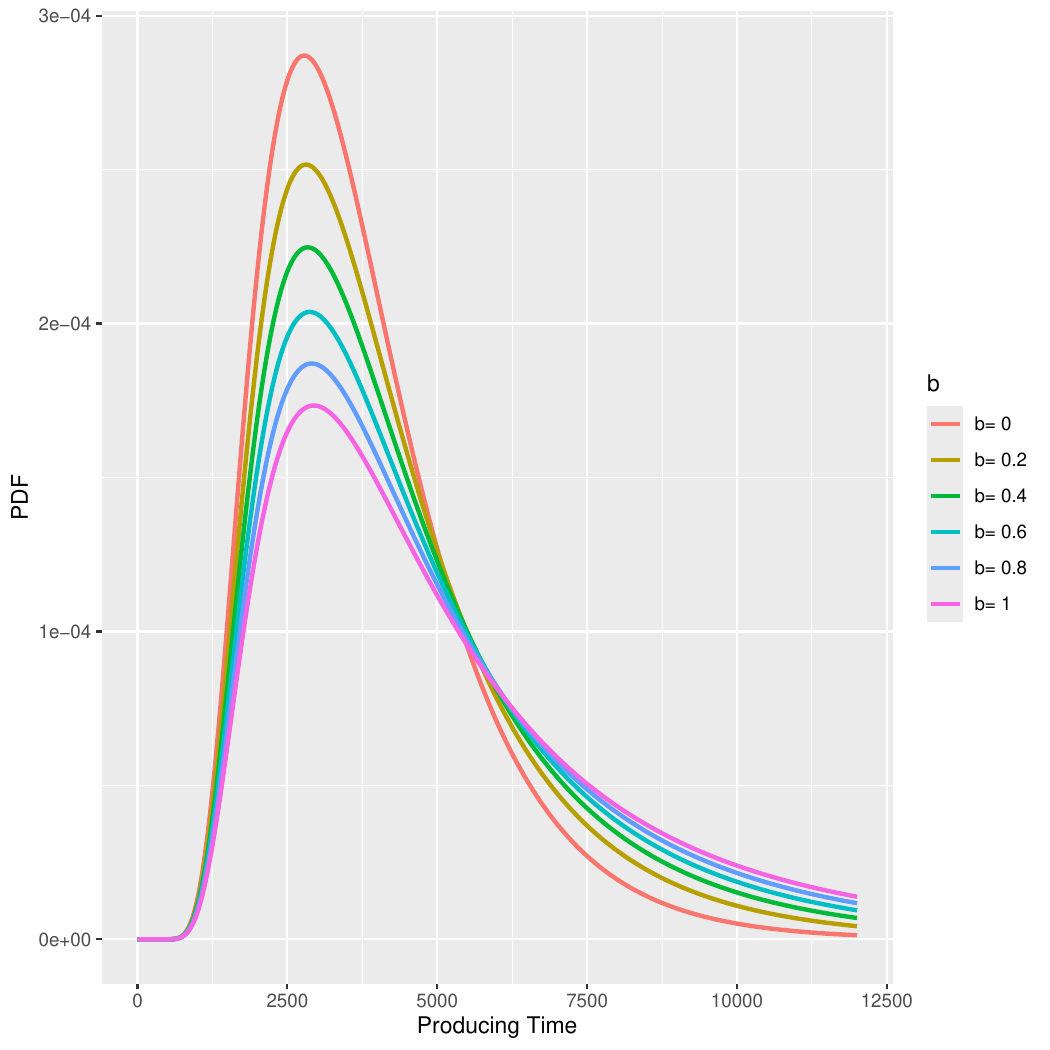}  
\end{center}
\caption{Approximation of the pdf of the hitting time for the second model}
\label{fig:ht2}
\end{figure}

\section{Conclusion}

We introduce here two stochastic versions of the Arps equation based on scalar linear SDEs. Solutions have been calculated explicitly. First passage time (FPT) distribution of a fixed level was also considered. Stochastic comparison properties show that the disadvantages pointed out in the literature when using the Arps equation hold also for these models. Further study on these stochastic models could be the statistical inference. In fact, in oil production context, available data are obtained rather from the oil production cumulative rate. It means that one does not observe $(Q_t)$ at several times, but the stochastic process $(Q_t^c)$ defined as follows:
$$
\forall t \geqslant 0 , \quad Q_t^c = \int_0^t Q_u \,\mathrm{d}u .
$$
On Figure \ref{fig:csde} we present simulations of $(Q_t^c)$ from the two stochastic models we studied here.
\begin{figure}
\begin{center}
\begin{tabular}{cc}
\includegraphics[width=6cm]{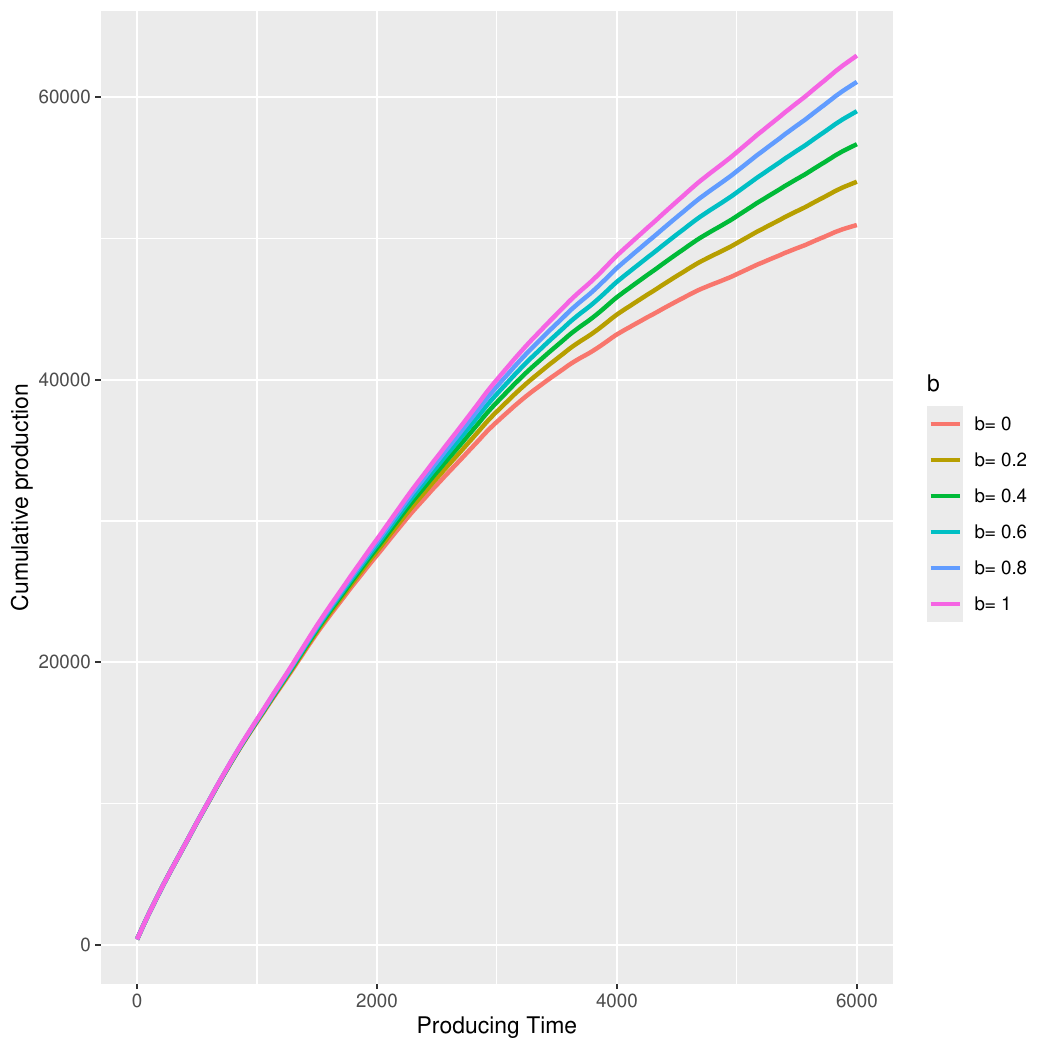} & 
\includegraphics[width=6cm]{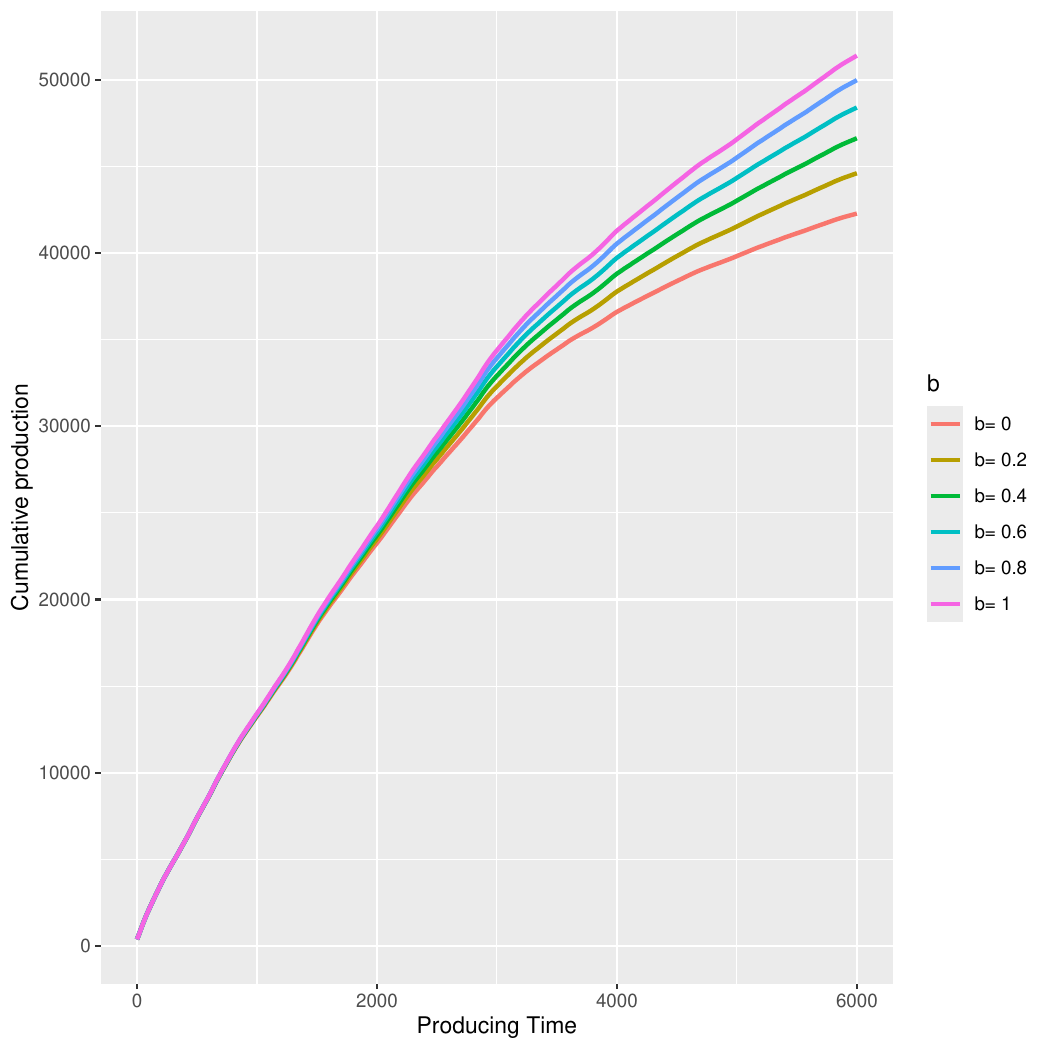}
\end{tabular}
\end{center}
\caption{Simulations of the cumulative stochastic models}
\label{fig:csde}
\end{figure}
Simulations on the left concerns the stochastic process $(Q_t^c)$ when $(Q_t)$ has a constant volatility while those on the right concerns the stochastic process $(Q_t^c)$ when $(Q_t)$ has a (positive) linear volatility. From the previous simulations of the two SDEs, it is not so surprisingly that sample paths on the left plot are more smooth than the ones of the right. If $(Q_t)$ is a Gaussian process, then $(Q_t^c)$ is still a Gaussian process and is a called integrated Gaussian process (this is clearly the case for the first stochastic Arps model we introduced). Statistical inference for such kind of models have been studied previously in the literature, see \cite{Gloter1} for instance. Beside, the approach developed here could be adapted to other models for flow behaviour of shale gas, like the Matthews‑Lefkovits model, the Generalized Weng’s prediction model and the Stretched Exponential Decline Production model, see \cite{Coutry} or \cite{Tan}.






\bibliographystyle{plain}

\end{document}